\documentclass[12pt]{article}
\usepackage{amsmath,amsfonts,amsthm,amssymb,mathrsfs,mathtools}
\usepackage{graphicx}
\usepackage[usenames,dvipsnames]{color}
\usepackage{float}
\usepackage{graphicx}
\usepackage{subfigure}
\usepackage{caption}
\usepackage[symbol]{footmisc}
\usepackage{cite}
\usepackage{url}
\usepackage{caption}
\usepackage{epstopdf}
\usepackage{hyperref}

\usepackage{color}

\theoremstyle{plain}
\newtheorem{theorem}{Theorem}[section]

\newtheorem{lemma}[theorem]{Lemma}
\newtheorem{proposition}[theorem]{Proposition}

\theoremstyle{definition}

\theoremstyle{remark}
\newtheorem{remark}[theorem]{Remark}

\numberwithin{equation}{section}
\numberwithin{theorem}{section}

\usepackage{a4wide}
\usepackage{geometry}
\begin{document}
	\renewcommand{\thefootnote}{\fnsymbol{footnote}}
	
	\begin{center}
		{\Large \textbf{Bridge-Type Processes Associated with L\'evy Processes and Their Decompositions}} \\[0pt]
		~\\[0pt] \textbf{Mohamed Erraoui} \footnote[1]{Mathematics Department, Faculty of Sciences, Choua\"ib Doukkali University, Route Ben Maachou, 24000, El jadida, Morocco. 
			E-mail: \texttt{erraoui@uca.ac.ma}} \textbf{Astrid Hilbert} \footnote[2]{Linnaeus University, Vejdesplats 7, SE-351 95 V\"axj\"o, Sweden.
			E-mail: \texttt{astrid.hilbert@lnu.se}} and \textbf{Mohammed Louriki} \footnote[3]{Mathematics Department, Faculty of Sciences Semalalia, Cadi Ayyad University, Boulevard Prince Moulay Abdellah, P. O. Box 2390, Marrakesh 40000, Morocco. E-mail: \texttt{m.louriki@uca.ac.ma}}
		\\[0pt]
	\end{center}
\begin{abstract}
We study a class of stochastic bridge-type processes whose terminal pinning value is random and is generated by an underlying stochastic process. In contrast with classical bridges, the construction depends not only on the terminal value of the driving process but also on its evolution before the terminal time. This dynamic stochastic input breaks some of the classical Markovian structure and requires a separate analysis of the semimartingale decomposition in the natural filtration.

We first analyze the Brownian case, which provides a Gaussian reference model, and show that the corresponding process is not Markovian in its natural filtration. We then extend the study to non-Gaussian L\'evy drivers, focusing on finite variation jump processes and on L\'evy processes with both Gaussian and jump components. In each case, we study the Doob--Meyer decomposition in the natural filtration.
\end{abstract}
\section{Introduction}
		\label{Setion_1}
\quad\,\,Brownian bridges play a central role in stochastic analysis, with applications ranging from statistical inference to mathematical finance; see, for instance, \cite{B}, \cite{R}, \cite{K}, \cite{BS}, \cite{EW}, and \cite{JY}. Classically, a Brownian bridge is defined as a Brownian motion conditioned to attain a fixed terminal value, usually called the pinning point. Its dynamics are governed by a singular linear drift term, which yields Markov and semimartingale properties together with an explicit canonical decomposition; see, e.g., \cite{A}, \cite{SY}, \cite{GSV}, \cite{RY}, and \cite{EY}. More recently, this framework has been extended to bridges with random pinning points, which have been shown to retain many of the structural properties of the classical Brownian bridge; see, e.g., \cite{BHM2007}, and \cite{L2}. Bridges with random length have also been studied; see, for instance, \cite{BBE},  and \cite{EHL(Levy)}.
	
Let $(W_t)_{t\in[0,T]}$ be a standard Brownian motion and let $Z$ be a random variable independent of $W$. A standard representation of the Brownian bridge $(\beta_t^T)_{t\in[0,T]}$ from $0$ to $Z$ is given by
$$ \beta_t^T= W_t-\frac{t}{T}W_T+\frac{t}{T}Z, \qquad t\in[0,T].$$
In the spirit of the information-based framework of Brody, Hughston, and Macrina; see, e.g., \cite{BHM2007}, this representation may be interpreted as a model of progressive information flow: the process $(\beta_t^T)_{t\in[0,T]}$ reveals the terminal variable $Z$ in a noisy and gradual manner, with the uncertainty vanishing as $t$ approaches $T$.
	
When the pinning point is itself the terminal value $X_T$ of a stochastic process $(X_t)_{t\in[0,T]}$, one may construct bridge-type processes whose behaviour depends not only on the terminal value but also on the temporal evolution of $X$. In this work, our main object of interest is the process
$$ \bar{X}_t = W_t-\frac{t}{T}W_T+\frac{t}{T}X_t, \qquad t\in[0,T], $$
where $(X_t)_{t\in[0,T]}$ is a L\'evy process independent of the Brownian motion $W$. 

The process $\left(\frac{t}{T}X_t\right)_{t\in[0,T]}$ may be interpreted as a time-dependent observation which progressively reveals information about the underlying process $X$. The factor $t/T$ acts as an observation scale: at the initial time no information is added, whereas at the terminal time the additional component coincides with the terminal value $X_T$. Hence, in contrast with the classical Brownian bridge with a fixed random pinning point, the random contribution is not determined solely by a terminal variable known from the outset, but evolves dynamically with the driving process.

Since Brownian motion is a particular example of a L\'evy process, the Brownian case naturally appears as a fundamental special case of this general framework. For this reason, we begin our analysis with the case where $X$ is a Brownian motion. This Gaussian setting is of independent interest, since it allows explicit computations and provides a useful reference model for the general L\'evy-driven construction.

The main objective of this paper is to investigate the structural properties of these bridge-type processes in their natural filtrations. In the Brownian case, we show that the process is not Markovian with respect to its natural filtration. This also indicates that one should not expect a general Markov property in the broader L\'evy framework. The principal focus of the paper is therefore on the semimartingale structure. We establish the semimartingale property and derive the corresponding Doob--Meyer decompositions in the natural filtrations. The analysis is carried out first in the Gaussian case and then beyond the Gaussian framework, for finite variation L\'evy drivers and for L\'evy processes with both Gaussian and jump components.
    
Let us note that in \cite{FWY}, the authors study the canonical decomposition of Brownian bridges between two Brownian motions. More precisely, they consider a bridge-type stochastic differential equation in which the fixed pinning point appearing in the drift of the classical Brownian bridge is replaced by an independent Brownian motion. Our approach differs from \cite{FWY} in several respects. First, we rely on a different representation of the bridge, namely an explicit bridge-type representation involving the scaled process $\frac{t}{T}X_t$. Second, our analysis is not restricted to the Gaussian framework: after treating the Brownian case as a guiding example, we extend the study to the case where the driving process is a general L\'evy process.

The remainder of the paper is organized as follows. In Section~2, we study the Brownian case, where the L\'evy driver is replaced by an independent Brownian motion. We compute the covariance function, show that the resulting bridge-type process is not Markovian in its natural filtration, and derive its semimartingale and Doob--Meyer decompositions. In Section~3, we turn to the non-Gaussian L\'evy framework. We first treat finite variation L\'evy drivers and then consider L\'evy processes with both Gaussian and finite variation jump components. In both settings, we establish the semimartingale property and derive the corresponding Doob--Meyer decompositions in the natural filtrations.      

	\section{The Brownian Case: A Gaussian Bridge-Type Process} In this section, we study the bridge-type process \begin{equation} 
    \bar{\beta}^{T}_t = W_t - \frac{t}{T}W_T + \frac{t}{T} B_t, \qquad 0 \leq t \leq T, \label{betabar} \end{equation} 
    where $(W_t)_{0\leq t\leq T}$ and $(B_t)_{0\leq t\leq T}$ are two independent standard Brownian motions. This process may be viewed as the Brownian counterpart of the more general L\'evy-driven bridge-type process considered later in the paper. Our aim in this section is to establish the basic properties of $\bar{\beta}^{T}$ with respect to its natural filtration. In particular, we study its Gaussian structure, compute its covariance function, analyze its Markov property, and derive its semimartingale and Doob--Meyer decompositions. These results serve as a reference point for the general L\'evy case, where the Gaussian structure is no longer available.\\
	It is immediate that \(\bar{\beta}^T\) is a centred Gaussian process with covariance function given by
	\begin{equation}
		\mathbb{E}\!\left[\bar{\beta}^T_s\, \bar{\beta}^T_t\right]
		= s \wedge t
		\;-\; \frac{st}{T}\!\left(1 - \frac{s \wedge t}{T}\right),
		\qquad s,t \in [0,T].
		\label{eqcovfunctionbar}
	\end{equation}
	\begin{proposition}
		The process $\bar{\beta}^T$ is not an $\mathbb{F}^{\bar{\beta}^T}$-Markov process.
	\end{proposition}
	\begin{proof}
		The fact that is $\bar{\beta}^T$ is not a Markov process can be deduced from the well-known characterization of the Markov property in the Gaussian setting (see e.g.  Ch.III p. 86 in \cite{RY}). Indeed, for $s<u\leq T$ and $t=(s+u)/2$ we have
		\begin{equation}
			\mathbb{E}[\bar{\beta}^T_s\bar{\beta}^T_t]\,\mathbb{E}[\bar{\beta}^T_t\bar{\beta}^T_u]\neq \mathbb{E}[(\bar{\beta}^T_t)^2]\,\mathbb{E}[\bar{\beta}_s^T\bar{\beta}_u^T].\label{eqgaussianMarkov}
		\end{equation}
		Which is the desired result.
	\end{proof}
	Concerning the semimartingale property, observing that the processes $(\frac{t}{T}\,B_t,\,t\geq 0)$ and $\beta^T$ are both semi-martingales with respect to $\mathbb{F}^{B}\vee\mathbb{F}^{\beta^T}$, then the process $\bar{\beta}^T$ (defined in \eqref{betabar}) is a semi-martingale with respect to $\mathbb{F}^{B}\vee\mathbb{F}^{\beta^T}$. Hence, Stricker's theorem (see, e.g., \cite[Ch.II, Theorem 4]{P}) yields that the process $\bar{\beta}^T$ is an $\mathbb{F}^{\bar{\beta}^T}$-semi-martingale. However, exploiting the properties of both the Brownian motion $B$ and the Brownian bridge $\beta^T$, direct calculation reveals that
	\begin{equation}
		\mathbb{E}[\bar{\beta}^T_t\vert \mathcal{F}^{B}_{s}\vee\mathcal{F}^{\beta^T}_{s}]=\dfrac{T-t}{T-s}\bar{\beta}^T_s+\sigma\,T\dfrac{t-s}{T-s}B_s,\,\,0\leq s\leq t<T,\label{eqcodexpbartgivenFBbeta}
	\end{equation}
	which implies that $\bar{\beta}^T$ is not a martingale with respect to $\mathbb{F}^{B}\vee\mathbb{F}^{\beta^T}$. Since the process $\bar{\beta}^T$ is a centred Gaussian process with covariance function given in \eqref{eqcovfunctionbar}, one might wonder whether or not $\bar{\beta}^T$ is a quasi-martingale? 
	
	Before providing an answer to this question, we briefly remind the reader of the definition of the quasi-martingale. An $\mathbb{F}$-adapted, c\`adl\`ag process $Z$ is a quasi-martingale on $[0, T]$ with respect to $\mathbb{F}$ if $\mathbb{E}[\vert Z_t\vert]<+\infty$, for each $t$, and if $$\sup\limits_{\mathcal{P}}\sum\limits_{i=1}^{n-1}\mathbb{E}\bigg[\bigg\vert \mathbb{E}\bigg[ Z_{t_{i+1}}-Z_{t_{i}}\bigg\vert \mathcal{F}_{t_i} \bigg] \bigg\vert\bigg]<+\infty,$$
	where $\mathcal{P}$ is the set of all finite partitions $0=t_0<t_1<\ldots<t_{n-1}<t_n=T$ of $[0, T]$. For a deeper discussion of quasi-martingale concepts we refer the reader to a section in chapter III in Protter \cite{P} which is devoted to the study of quasi-martingales. We have the following result:
	\begin{proposition}\label{propquasimartingale}
		The process $\bar{\beta}^T$ is a quasi-martingale with respect to $\mathbb{F}^{\beta^T}\vee\mathbb{F}^{B}$.
	\end{proposition}
	\begin{proof}
		We need to show that: 
		\begin{enumerate}
			\item[(i)] For all $0\leq t\leq T$, $\mathbb{E}[\vert \bar{\beta}^T_t \vert]<+\infty$.
			\item[(ii)] $\sup\limits_{\mathcal{P}}\sum\limits_{i=1}^{n-1}\mathbb{E}\bigg[\bigg\vert \mathbb{E}\bigg[ \bar{\beta}^T_{t_{i+1}}-\bar{\beta}^T_{t_{i}}\bigg\vert \mathcal{F}^{B}_{t_i}\vee\mathcal{F}^{\beta^T}_{t_i} \bigg] \bigg\vert\bigg]<+\infty,$
			where $\mathcal{P}$ is the set of all finite partitions $0=t_0<t_1<\ldots<t_{n-1}<t_n=T$ of $[0, T]$.
		\end{enumerate}
		As regards statement (i), it is easy to see that $\mathbb{E}[\vert \bar{\beta}^T_t \vert]=\sqrt{\dfrac{2t}{\pi T}\bigg(T-t+\dfrac{t^2}{T}\bigg)}<+\infty$. Concerning the statement (ii), we have for $1\leq i\leq n$
		\begin{align*}
			\mathbb{E}\bigg[ \bar{\beta}^T_{t_{i+1}}-\bar{\beta}^T_{t_{i}}\bigg\vert \mathcal{F}^{B}_{t_i}\vee\mathcal{F}^{\beta^T}_{t_i} \bigg]&=\mathbb{E}\bigg[ \beta^T_{t_{i+1}}+\dfrac{t_{i+1}}{T}(B_{t_{i+1}}-B_{t_i})\bigg\vert \mathcal{F}^{\beta^T}_{t_i} \bigg]-\beta^T_{t_{i}}+\dfrac{t_{i+1}-t_i}{T}B_{t_i}\\
			&=-\dfrac{t_{i+1}-t_{i}}{T-t_i}\beta^T_{t_{i}}+\dfrac{t_{i+1}-t_i}{T}B_{t_i}.
		\end{align*}
		Hence,
		\begin{align*}
			\mathbb{V}\text{ar}\left(\mathbb{E}\bigg[ \bar{\beta}^T_{t_{i+1}}-\bar{\beta}^T_{t_{i}}\bigg\vert \mathcal{F}^{B}_{t_i}\vee\mathcal{F}^{\beta^T}_{t_i} \bigg]\right)&=\dfrac{(t_{i+1}-t_{i})^2}{(T-t_i)^2}\dfrac{t_i(T-t_i)}{T}+\dfrac{ (t_{i+1}-t_i)^2}{T^2}t_i\\
			&=(t_{i+1}-t_{i})^2t_i\left[\dfrac{1}{T(T-t_i)}+\dfrac{1}{T^2}\right]
		\end{align*}
	This implies that
		\begin{align}
			\mathbb{E}\bigg[\bigg\vert \mathbb{E}\bigg[ \bar{\beta}^T_{t_{i+1}}-\bar{\beta}^T_{t_{i}}\bigg\vert \mathcal{F}^{B}_{t_i}\vee\mathcal{F}^{\beta^T}_{t_i} \bigg] \bigg\vert\bigg]&=\sqrt{\dfrac{2}{\pi}}\bigg(t_{i+1}-t_{i}\bigg)\sqrt{\dfrac{t_{i}(2T-t_i)}{T^2(T-t_i)}}\leq\dfrac{2}{\sqrt{\pi}}\dfrac{t_{i+1}-t_{i}}{\sqrt{T-t_i}}.\label{eqquasi-martingale}
		\end{align}
		It follows from \eqref{eqquasi-martingale} that 
		\begin{equation}
			\sum\limits_{i=1}^{n-1}\mathbb{E}\bigg[\bigg\vert \mathbb{E}\bigg[ \bar{\beta}^T_{t_{i+1}}-\bar{\beta}^T_{t_{i}}\bigg\vert \mathcal{F}^{B}_{t_i}\vee\mathcal{F}^{\beta^T}_{t_i} \bigg] \bigg\vert\bigg]\leq \dfrac{2}{\sqrt{\pi}}\sum\limits_{i=1}^{n-1}\dfrac{t_{i+1}-t_{i}}{\sqrt{T-t_i}}.\label{equpperbound}
		\end{equation}
		The right hand side in \eqref{equpperbound} represents exactly the lower Darboux sum for the function $\dfrac{2}{\sqrt{\pi(T-x)}}$ on the interval $[0, T]$. Thus,
		\begin{equation}
			\sup\limits_{\mathcal{P}}\sum\limits_{i=1}^{n-1}\mathbb{E}\bigg[\bigg\vert \mathbb{E}\bigg[ \bar{\beta}^T_{t_{i+1}}-\bar{\beta}^T_{t_{i}}\bigg\vert \mathcal{F}^{B}_{t_i}\vee\mathcal{F}^{\beta^T}_{t_i} \bigg] \bigg\vert\bigg]\leq \dfrac{2}{\sqrt{\pi}}\sup\limits_{\mathcal{P}}\sum\limits_{i=1}^{n-1}\dfrac{t_{i+1}-t_{i}}{\sqrt{T-t_i}}\leq 4\sqrt{\dfrac{T}{\pi}}<\infty,
		\end{equation}
		which is the desired result.
	\end{proof}
	
	Since $\bar{\beta}^T$ is an $\mathbb{F}^{\bar{\beta}^T}$-semi-martingale, a natural question arises: what is the explicit form of its canonical decomposition? In the following results we provide an answer to this question.
\begin{lemma}
		The process $(\bar{M}_t, 0\leq t \leq T)$ given by 
		\begin{equation}
			\bar{M}^T_t=\bar{\beta}^T_t-\displaystyle\int_{0}^{t}\dfrac{ B_s-\bar{\beta}^T_s}{T-s}\mathrm{d}s,\label{eqmamrtingale_barM}
		\end{equation}
		is a Gaussian martingale with respect to $\mathbb{F}^{\beta^T}\vee\mathbb{F}^{B}$.
\end{lemma}
\begin{proof}
	We have for all $0\leq t\leq T$,
\begin{equation}
	\mathbb{E}\bigg[\displaystyle\int_0^t\bigg\vert \dfrac{\bar{\beta}^T_s-B_s}{T-s} \bigg\vert\mathrm{d}s\bigg]\leq \displaystyle\int_0^t \bigg(\dfrac{\mathbb{E}[\vert\beta^T_s\vert]}{T-s}+\dfrac{\mathbb{E}[\vert B_s\vert]}{T}\bigg)\mathrm{d}s=\sqrt{\dfrac{2}{T\pi}}\displaystyle\int_0^t\bigg( \sqrt{\dfrac{s}{T-s}}+\sqrt{\dfrac{s}{T}}\bigg) \mathrm{d}s<\infty.\label{eqintegralfinite}
\end{equation}
	Hence, the integral at the right-hand side of \eqref{eqmamrtingale_barM} is well defined for $0\leq t\leq T$. On the other hand, for all $s< t \leq T$, we have
	\begin{align}
		\mathbb{E}[\bar{M}^T_t-\bar{M}^T_s\vert \mathcal{F}^{B}_{s}\vee\mathcal{F}^{\beta^T}_{s}]=	\mathbb{E}[\bar{\beta}^T_t-\bar{\beta}^T_s\vert \mathcal{F}^{B}_{s}\vee\mathcal{F}^{\beta^T}_{s}]-\displaystyle\int_{s}^{t}\mathbb{E}\bigg[\dfrac{B_u-\bar{\beta}^T_u}{T-u}\bigg\vert \mathcal{F}^{B}_{s}\vee\mathcal{F}^{\beta^T}_{s} \bigg]\mathrm{d}u.\label{eqShowmartingalegivenFBbeta}
	\end{align}
	It follows from \eqref{eqcodexpbartgivenFBbeta} that
	\begin{equation}
		\mathbb{E}[\bar{\beta}^T_t-\bar{\beta}^T_s\vert \mathcal{F}^{B}_{s}\vee\mathcal{F}^{\beta^T}_{s}]=-\dfrac{t-s}{T-s}\bar{\beta}^T_s+\dfrac{t-s}{T-s}B_s\label{eqEgivenBeta_B1}
	\end{equation}
	and
	\begin{equation}
		\mathbb{E}[B_u-\bar{\beta}^T_u\vert \mathcal{F}^{B}_{s}\vee\mathcal{F}^{\beta^T}_{s}]=-\dfrac{T-u}{T-s}\bar{\beta}^T_s+\dfrac{T-u}{T-s}B_s.\label{eqEgivenBeta_B2}
	\end{equation}
	Inserting \eqref{eqEgivenBeta_B1} and \eqref{eqEgivenBeta_B2} into \eqref{eqShowmartingalegivenFBbeta} we conclude that for all $s\leq t <T$,
	\begin{align*}
		\mathbb{E}[\bar{M}^T_t\vert \mathcal{F}^{B}_{s}\vee\mathcal{F}^{\beta^T}_{s}]=\bar{M}^T_s.
	\end{align*}
	This completes the proof.
\end{proof}
	In the next result, we provide the semi-martingale decomposition of $\bar{\beta}^T$ with respect to its natural filtration.
	\begin{proposition}
		The canonical decomposition of $\bar{\beta}^T$ in its natural filtration $\mathbb{F}^{\bar{\beta}^T}$ is given by
		\begin{equation}
			\bar{\beta}^T_t=\displaystyle\int_{0}^{t}\sqrt{\frac{T^2+s^2}{T^2}}\mathrm{d}\bar{B}_s-\displaystyle\int_{0}^{t}\dfrac{\bar{\beta}^T_s}{T-s}\mathrm{d}s+\displaystyle\int_{0}^{t}\displaystyle\int_{0}^{s}\dfrac{\mathfrak{a}(s,u)}{T-s}\mathrm{d}\bar{\beta}^T_u\mathrm{d}s,\,\,t<T,\label{eqcanonivaldecompositionbarbeta}
		\end{equation}
		where $\bar{B}$ is an $\mathbb{F}^{\bar{\beta}^T}$-Brownian motion, and for $0\leq u\leq s<T,$
		\begin{equation}
			\mathfrak{a}(s,u)=\dfrac{T-s}{T-s+ s\,\tan^{-1}(\frac{s}{T})}\bigg(\dfrac{T u}{u^2+T^2}+\tan^{-1}\left(\frac{u}{T}\right)\bigg)+\dfrac{s\,\tan^{-1}(\frac{s}{T})}{T-s+s\,\tan^{-1}(\frac{s}{T})}.\label{eqa(s,u)}
		\end{equation}
	\end{proposition}
	\begin{proof}
		By \eqref{eqintegralfinite} and a well known result of filtering theory, see for instance Theorem 8.1.1 and Remark 8.1.1 in \cite{Kall} or  Proposition 2.30, p. 33 in \cite{BCfiltering}, it follows from \eqref{eqmamrtingale_barM} that the process
		\begin{equation}
			\bar{m}_t:=\bar{\beta}^T_t+\displaystyle\int_{0}^{t}\dfrac{\bar{\beta}^T_s-\mathbb{E}[B_s\vert \mathcal{F}_s^{\bar{\beta}^T}]}{T-s}\mathrm{d}s,\, t<T\label{eqbarm}
		\end{equation}
		is an $\mathbb{F}^{\bar{\beta}^T}$-martingale, which is Gaussian by construction. Due to \eqref{eqbarm} and Theorem 1 in \cite{MH} we may assume that 
		\begin{equation}
			\mathbb{E}[B_t\vert \mathcal{F}_t^{\bar{\beta}^T}]=\displaystyle\int_{0}^{t}\mathfrak{a}(t,u)\mathrm{d}\bar{\beta}^T_u.\label{eqconexpBargicenFbar}
		\end{equation}
		Therefore, we have only to derive the explicit form of the deterministic function $\mathfrak{a}$. Using the projection property of the conditional expectation, we have for all $0\leq s\leq t<T$,
		\begin{equation*}
			\mathbb{E}[\bar{\beta}^T_s(B_t-\mathbb{E}[B_t\vert \mathcal{F}_t^{\bar{\beta}^T}])]=0.
		\end{equation*}
		Hence, we have 
		\begin{equation*}
			\mathbb{E}[\bar{\beta}^T_sB_t]-\mathbb{E}[\bar{\beta}^T_s\bar{\beta}^T_t]\,\mathfrak{a}(t,t)+\displaystyle\int_{0}^{t}\mathbb{E}[\bar{\beta}^T_s\bar{\beta}^T_u]\dfrac{\partial \mathfrak{a}(t,u)}{\partial u}\mathrm{d}u=0.
		\end{equation*}
		Using \eqref{eqcovfunctionbar}, we get 
		\begin{equation*}
		\dfrac{s^2}{T}-\mathfrak{a}(t,t)\bigg[s-\dfrac{st}{T}+\dfrac{s^2t}{T^2}\bigg]+\displaystyle\int_{0}^{s}\bigg[u-\dfrac{us}{T}+\dfrac{u^2s}{T^2}\bigg]\dfrac{\partial \mathfrak{a}(t,u)}{\partial u}\mathrm{d}u+\displaystyle\int_{s}^{t}\bigg[s-\dfrac{su}{T}+\dfrac{s^2u}{T^2}\bigg]\dfrac{\partial \mathfrak{a}(t,u)}{\partial u}\mathrm{d}u=0.
		\end{equation*}
		By taking the first derivatives with respect to $s$ we obtain
		\begin{equation}
			2\dfrac{s}{T}-\mathfrak{a}(t,t)\bigg[1-\dfrac{t}{T}+2\dfrac{st}{T^2}\bigg]+\displaystyle\int_{0}^{s}\bigg[\dfrac{u^2}{T^2}-\dfrac{u}{T}\bigg]\dfrac{\partial \mathfrak{a}(t,u)}{\partial u}\mathrm{d}u+\displaystyle\int_{s}^{t}\bigg[1-\dfrac{u}{T}+2\dfrac{su}{T^2}\bigg]\dfrac{\partial \mathfrak{a}(t,u)}{\partial u}\mathrm{d}u=0.\label{eqfirstderivatives}
		\end{equation}
		Taking further derivatives with respect to $s$ we obtain
		\begin{equation}
			2T -2  s \,\mathfrak{a}(t,s) - (s^2+T^2)\dfrac{\partial \mathfrak{a}(t,s)}{\partial s}
			-2\displaystyle\int_{s}^{t}\mathfrak{a}(t,u)\mathrm{d}u=0,\label{eqsecondderivatives}
		\end{equation}
		and
		\begin{equation}
			4s\,\dfrac{\partial \mathfrak{a}(t,s)}{\partial s}+(T^2+s^2)\,\dfrac{\partial^2 \mathfrak{a}(t,s)}{\partial^2 s}=0.\label{eqthirdderivatives}
		\end{equation}
		This implies that $\mathfrak{a}(t,s)$ is of the form
		\begin{equation}
			\mathfrak{a}(t,s)=c(t)\bigg[\dfrac{s}{T^2(s^2+T^2)}+\frac{1}{T^3}\tan^{-1}\left(\frac{s}{T}\right)\bigg]+d(t).\label{eqa(t,s)}
		\end{equation}
		Substituting \eqref{eqa(t,s)} in  \eqref{eqsecondderivatives}, and taking $s=t$, we obtain 
		\begin{equation}
			2T -2 t \,a(t,t) -2 \dfrac{c(t)}{T^2+t^2}=0,\label{eqsecondderivatives1}
		\end{equation}
		hence,
		\begin{equation}
			2T-2 \frac{t}{T^3}\tan^{-1}\left(\frac{t}{T}\right)\,c(t)-2 t \,d(t)- 2\frac{1}{T^2}c(t)=0.\label{eqsecondderivatives1}
		\end{equation}
		Substituting \eqref{eqa(t,s)} in \eqref{eqfirstderivatives}, and taking $s=0$, we obtain 
		\begin{equation}
			-\mathfrak{a}(t,t)\bigg[1-\dfrac{t}{T}\bigg]+\displaystyle\int_{0}^{t}\bigg[1-\dfrac{u}{T}\bigg]\dfrac{\partial \mathfrak{a}(t,u)}{\partial u}\mathrm{d}u=0.\label{eqfirstderivatives1}
		\end{equation}
		hence,
		\begin{multline}
			-\dfrac{T-t}{T}\left(\dfrac{tc(t)}{T^2(t^2+T^2)}-c(t)\frac{1}{T^3}\tan^{-1}(\frac{t}{T})-d(t)\right) \\+\frac{1}{T^3} c(t)\left(\dfrac{ t(T-t)}{t^2+T^2}+\tan^{-1}\left( \frac{t}{T}\right)\right)=0\label{eqfirstderivatives11}
		\end{multline}
		which implies that
		\begin{equation}
			d(t)= c(t) \dfrac{t}{T^3(T-t)}\tan^{-1}\left( \frac{t}{T}\right).\label{eqfirstderivatives111}
		\end{equation}
		Inserting \eqref{eqfirstderivatives111} into \eqref{eqsecondderivatives1}, we obtain
		\begin{equation*}
			c(t)=\dfrac{T^3(T-t)}{T-t+ t\,\tan^{-1}(\sigma t)}.
		\end{equation*}
		This latter and \eqref{eqfirstderivatives111} yield
		\begin{equation*}
			d(t)=\dfrac{t\,\tan^{-1}\left(\frac{t}{T}\right)}{T-t+t\,\tan^{-1}\left(\frac{t}{T}\right)}.
		\end{equation*}
		Finally, since $\bar{m}$ is a Gaussian martingale of finite quadratic variation such that $\langle \bar{m} \rangle_t=t+\dfrac{t^3}{3T^2}$, $\bar{m}$ can be represented as 
		\begin{equation*}
			\bar{m}_t=\displaystyle\int_{0}^{t}\sqrt{\frac{T^2+s^2}{T^2}}\mathrm{d}\bar{B}_s,	
		\end{equation*}
		where $\bar{B}$ is an $\mathbb{F}^{\bar{\beta}^T}$-Brownian motion. This completes the proof.
	\end{proof}
	\begin{remark}
		As a consequence of \eqref{eqcodexpbartgivenFBbeta} and \eqref{eqconexpBargicenFbar}, for $s\leq t<T$, the conditional expectation of $\bar{\beta}^T_{t}$ given $\mathcal{F}^{\bar{\beta}^T}_{s}$ by
		\begin{multline*}
			\mathbb{E}[\bar{\beta}^T_{t}\vert \mathcal{F}^{\bar{\beta}^T}_{s}]=\left[\dfrac{T-t}{T-s}+\dfrac{s(t-s)\,\tan^{-1}\left(\frac{s}{T}\right)}{\left(T-s+ s\,\tan^{-1}\left(\frac{s}{T}\right)\right)^2}\right]\bar{\beta}^T_s\\+\dfrac{t-s}{T-s+ s\,\tan^{-1}\left(\frac{s}{T}\right)}\displaystyle\int_0^s\bigg(\dfrac{T u}{u^2+T^2}+\tan^{-1}\left(\frac{u}{T}\right)\bigg)\mathrm{d}\bar{\beta}^T_u.
		\end{multline*}
		This perfectly underlines the loss of the Markovian property.
	\end{remark}
\begin{proposition}
	For any $0<t<T$, we have
	\begin{equation}
		\mathbb{E}[B_{T}\vert \mathcal{F}^{\bar{\beta}^T}_{t}]=\displaystyle\int_0^t\frac{s (T - s) + \left( T^2+   s^2\right) \tan^{-1}\left(\frac{s}{T}\right)}{\sqrt{T^2 +  s^2} \left(T - s +  s \tan^{-1}\left(\frac{s}{T}\right)\right)}\mathrm{d}\bar{B}_s.
	\end{equation}
\end{proposition}
\begin{proof}
	We have for $0<t<T$,
	\begin{equation*}
		\mathbb{E}[B_{T}\vert \mathcal{F}^{\bar{\beta}^T}_{t}]=\mathbb{E}[\mathbb{E}[B_{T}\vert \mathcal{F}^{B}_{t}\vee\mathcal{F}^{\beta^T}_{t}]\vert \mathcal{F}^{\bar{\beta}^T}_{t}]=\mathbb{E}[\mathbb{E}[B_{T}\vert \mathcal{F}^{B}_{t}]\vert \mathcal{F}^{\bar{\beta}^T}_{t}]=\mathbb{E}[B_{t}\vert \mathcal{F}^{\bar{\beta}^T}_{t}].
	\end{equation*}
	It follows from \eqref{eqgaussianMarkov} that 
	\begin{equation*}
		\mathbb{E}[B_T\vert \mathcal{F}_t^{\bar{\beta}^T}]=\displaystyle\int_{0}^{t}\mathfrak{a}(t,u)\mathrm{d}\bar{\beta}^T_u
	\end{equation*}
	where,
\begin{equation*}
	\mathfrak{a}(s,u)=\dfrac{T-s}{T-s+ s\,\tan^{-1}(\frac{s}{T})}\bigg(\dfrac{T u}{u^2+T^2}+\tan^{-1}\left(\frac{u}{T}\right)\bigg)+\dfrac{s\,\tan^{-1}(\frac{s}{T})}{T-s+s\,\tan^{-1}(\frac{s}{T})}.
\end{equation*}
Hence,
	\begin{multline}
		\mathbb{E}[B_{T}\vert \mathcal{F}^{\bar{\beta}^T}_{t}]=\dfrac{T-t}{T-t+t\,\tan^{-1}(\frac{t}{T})}\displaystyle\int_{0}^{t}\left(\dfrac{Tu}{T^2+u^2}+\tan^{-1}\Big(\frac{u}{T}\Big)\right)\mathrm{d}\bar{\beta}^T_u+\\\dfrac{t\,\tan^{-1}(\frac{t}{T})}{T-t+t\,\tan^{-1}(\frac{t}{T})}\bar{\beta}^T_t.
	\end{multline}
	Setting $B_t^T=\mathbb{E}[B_{T}\vert \mathcal{F}^{\bar{\beta}^T}_{t}], 0<t<T$, by It\^o formula we obtain
	\begin{multline*}
		\mathrm{d}B_t^T=\frac{\sigma t (T - t) + T\left( 1+ \sigma^2  t^2\right) \tan^{-1}(\sigma t)}{\left(1 + \sigma^2 t^2\right) \left(T - t + \sigma T t \tan^{-1}(\sigma t)\right)}\mathrm{d}\bar{\beta}^T_t+\\
	\left(\bar{\beta}^T_t-\sigma T\displaystyle\int_{0}^{t}\left(\dfrac{\sigma u}{(\sigma^2u^2+1)}+\tan^{-1}(\sigma u)\right)\mathrm{d}\bar{\beta}^T_u\right)\frac{\sigma t (T - t) + T\left( 1+ \sigma^2  t^2\right) \tan^{-1}(\sigma t)}{\left(1 + \sigma^2 t^2\right) \left(T - t + \sigma T t \tan^{-1}(\sigma t)\right)^2}\mathrm{d}t.
	\end{multline*}
	From \eqref{eqcanonivaldecompositionbarbeta}, for $t<T$, we have
	\begin{multline*}
		\mathrm{d}\bar{\beta}^T_t=\sqrt{1+\sigma^2t^2}\,\mathrm{d}\bar{B}_t+\frac{\sigma T}{T - t + \sigma T t \tan^{-1}(\sigma t)}\displaystyle\int_{0}^{t}\left(\dfrac{\sigma u}{(\sigma^2u^2+1)}+\tan^{-1}(\sigma u)\right)\mathrm{d}\bar{\beta}^T_u\mathrm{d}t\\
		-\dfrac{\bar{\beta}^T_t}{T-t}\mathrm{d}t+\frac{\sigma  T t \tan^{-1}(\sigma t)}{\left(T-t\right) \left(T - t + \sigma T t \tan^{-1}(\sigma t)\right)}\bar{\beta}^T_t\mathrm{d}t\\
		=\sqrt{1+\sigma^2t^2}\,\mathrm{d}\bar{B}_t-\frac{1}{T - t + \sigma T t \tan^{-1}(\sigma t)}\left(\bar{\beta}^T_t-\sigma T\displaystyle\int_{0}^{t}\left(\dfrac{\sigma u}{(\sigma^2u^2+1)}+\tan^{-1}(\sigma u)\right)\mathrm{d}\bar{\beta}^T_u\right)\mathrm{d}t.
	\end{multline*}
	Thus,
	\begin{equation}
		\mathrm{d}B_t^T=\frac{\sigma t (T - t) + T\left( 1+ \sigma^2  t^2\right) \tan^{-1}(\sigma t)}{\sqrt{1 + \sigma^2 t^2} \left(T - t + \sigma T t \tan^{-1}(\sigma t)\right)}\mathrm{d}\bar{B}_t.
	\end{equation}
	This completes the proof.
\end{proof}
\section{The L\'evy Case: Beyond the Gaussian Framework}
The aim of this section is to extend the preceding analysis beyond the Gaussian setting by replacing the Brownian perturbation with a general L\'evy process. We consider the process
\begin{equation*}
	\bar{X}^T_t
	= W_t - \frac{t}{T} W_T + \frac{t}{T} X_t,
	\qquad 0 \le t \le T.
\end{equation*}
where $X$ is a L\'evy process independent of $W$. For background on Lévy processes, we refer the reader to the monographs \cite{Bertoin,Sato}. The analysis of the L\'evy-driven process depends on the structure of the
driving L\'evy process. We therefore separate the discussion into two steps. We first treat the finite variation case, where the jump structure can be described explicitly through the L\'evy--It\^o decomposition. We then add a Brownian component to the L\'evy driver, leading to a mixed Gaussian-jump framework and allowing a comparison with the purely Brownian case studied in the previous section. Since the Brownian case already shows that $\bar{\beta}^T$ is not Markovian in
its natural filtration, the general L\'evy-driven process cannot be Markovian in general. We therefore focus in this section on the semimartingale structure and the Doob--Meyer decomposition of $\bar{X}^T$.
\subsection{The finite variation L\'evy case}
We first consider the case where $X$ is a finite variation L\'evy process with characteristics $(\lambda,0,\nu)$, independent of $W$. By the L\'evy--It\^o decomposition, and under the finite variation condition we have
\begin{equation}
    X_t
    =
    \left(
        \lambda-\displaystyle\int_{|x|<1} x\,\nu(\mathrm{d}x)
    \right)t
    +
    \sum_{s\leq t}\Delta X_s,
    \qquad 0\leq t\leq T.\label{eqdecpurejumplevy}
\end{equation}
In this subsection, we study the corresponding bridge-type process
\begin{equation}
      \eta^T_t
    =
    W_t-\frac{t}{T}W_T+\frac{t}{T}X_t,
    \qquad 0\leq t\leq T.  \label{eqeta^T}
\end{equation}
We have for all $t$,
\begin{equation}
	\mathbb{E}[X_t]=\left(\lambda-\displaystyle\int_{\vert x\vert <1}  x\, \nu(\mathrm{d}x) \right) t+ \mathbb{E}\left[\sum\limits_{s\leq t}\Delta X_s\right]=t\left(\lambda +\displaystyle\int_{\vert x\vert\geq 1}x\,\nu(\mathrm{d}x)\right).\label{eqExp}
\end{equation}
Since the Brownian bridge component is continuous, the jumps of $\eta^T$ are exactly those of the scaled L\'evy component. More precisely, for every $t>0$,
\[
    \Delta \eta^T_t=\frac{t}{T}\Delta X_t.
\]
Consequently, the jump part of $X$ can be recovered from $\eta^T$. Using
\eqref{eqdecpurejumplevy}, we obtain
\begin{equation}
	\mathbb{F}^{\eta^T}=\mathbb{F}^{X}\vee\mathbb{F}^{\beta^T}.
\end{equation}
Since the processes $(t\,X_t,\,t\geq 0)$ and $\beta^T$ are both semi-martingales with respect to $\mathbb{F}^{X}\vee\mathbb{F}^{\beta^T}$, then the process $\eta^T$ is a semi-martingale with respect to $\mathbb{F}^{X}\vee\mathbb{F}^{\beta^T}$ and hence with respect to its own filtration.
\begin{proposition}
	For every $0\leq s\leq t<T$, we have
	\begin{multline*}
		\mathbb{E}[\eta^T_t\vert \mathcal{F}^{\eta^T}_{s}]=\dfrac{T-t}{T-s}\eta^T_s+ T\dfrac{t-s}{T-s}\sum\limits_{0<u\leq s}\dfrac{\Delta \eta^T_{u}}{u}\\
		+ \dfrac{t-s}{T}\left[\left(t+\frac{Ts}{T-s}\right)\lambda+t\displaystyle\int_{\vert x\vert\geq 1}x\,\nu(\mathrm{d}x)-\frac{Ts}{T-s}\displaystyle\int_{\vert x\vert< 1}x\,\nu(\mathrm{d}x)\right].
	\end{multline*}
\end{proposition}
\begin{proof}
	From the independence hypothesis, we have
	\begin{align}
		\mathbb{E}[\eta^T_t\vert \mathcal{F}^{\eta^T}_{s}]=\mathbb{E}[\eta^T_t\vert \mathcal{F}^{X}_{s}\vee\mathcal{F}^{\beta^T}_{s}]=\mathbb{E}[\beta^T_t\vert \mathcal{F}^{\beta^T}_{s}]+\dfrac{t}{T}\mathbb{E}[X_t\vert \mathcal{F}^{X}_{s}].\label{eqsumXbeta}
	\end{align}
	Using the Markov property and the Gaussian property of $\beta^T$, we get
	\begin{equation}
		\mathbb{E}[\beta^T_t\vert \mathcal{F}^{\beta^T}_{s}]=\mathbb{E}[\beta^T_t\vert \beta^T_{s}]=\dfrac{T-t}{T-s}\beta^T_s.\label{eqexpbeta^T}
	\end{equation}
	Using the L\'evy property of $X$, \eqref{eqdecpurejumplevy}, and \eqref{eqExp} we obtain
	\begin{equation}
		\mathbb{E}[X_t\vert \mathcal{F}^{X}_{s}]=\left(\lambda +\displaystyle\int_{\vert x\vert\geq 1}x\,\nu(\mathrm{d}x)\right)(t-s)+X_s.\label{eqexpX}
	\end{equation}
	Inserting \eqref{eqexpbeta^T} and \eqref{eqexpX} into \eqref{eqsumXbeta} we obtain
	\begin{equation}
		\mathbb{E}[\eta^T_t\vert \mathcal{F}^{\eta^T}_{s}]=\dfrac{T-t}{T-s}\,\eta^T_s+\dfrac{t-s}{T-s}X_s+\dfrac{t(t-s)}{T}\left(\lambda +\displaystyle\int_{\vert x\vert\geq 1}x\,\nu(\mathrm{d}x)\right).\label{eqEXtildebeta}
	\end{equation}
	The desired result follows from \eqref{eqdecpurejumplevy}.
\end{proof}
\begin{proposition}
	The process $(\tilde{M}_t, 0\leq t \leq T)$ given by 
	\begin{multline}
		\tilde{M}^T_t=\eta^T_t-\displaystyle\int_{0}^{t}\dfrac{\left(\lambda-\displaystyle\int_{\vert x\vert <1}  x\, \nu(\mathrm{d}x)\right)s+ T\sum\limits_{0<u\leq s}\frac{\Delta \eta^T_{u}}{u}-\eta^T_s}{T-s}\mathrm{d}s\\-\dfrac{t^2}{2T}\left(\lambda +\displaystyle\int_{\vert x\vert\geq 1}x\,\nu(\mathrm{d}x)\right),\label{eqmamrtingale_tildeM}
	\end{multline}
	is a martingale with respect to $\mathbb{F}^{\eta^T}$.
\end{proposition}
\begin{proof}
	We start the proof with the observation that the process $(\tilde{M}_t, 0\leq t \leq T)$ can be rewritten as
	\begin{equation}
		\tilde{M}^T_t=\eta^T_t-\displaystyle\int_{0}^{t}\dfrac{ X_s-\eta^T_s}{T-s}\mathrm{d}s-\dfrac{t^2}{2T}\left(\lambda +\displaystyle\int_{\vert x\vert\geq 1}x\,\nu(\mathrm{d}x)\right),\label{eqmamrtingale_tildeM2}
	\end{equation}
	We have for all $0\leq t\leq T$,
	\begin{equation}
		\mathbb{E}\bigg[\displaystyle\int_0^t\bigg\vert \dfrac{\eta^T_s-X_s}{T-s} \bigg\vert\mathrm{d}s\bigg]\leq \displaystyle\int_0^t \bigg(\dfrac{\mathbb{E}[\vert\beta^T_s\vert]}{T-s}+\dfrac{\mathbb{E}[\vert X_s\vert]}{T}\bigg)\mathrm{d}s\label{eqintfinite}
	\end{equation}
	with
	\begin{equation}
		\displaystyle\int_0^t \dfrac{\mathbb{E}[\vert\beta^T_s\vert]}{T-s}\mathrm{d}s=\sqrt{\dfrac{2}{T\pi}}\displaystyle\int_0^t \sqrt{\dfrac{s}{T-s}} \mathrm{d}s<\infty\label{eqfirstestimate}
	\end{equation}
	and
	\begin{equation}
		\displaystyle\int_0^t \mathbb{E}[\vert X_s\vert]\mathrm{d}s\leq \left(\vert\lambda\vert+2\displaystyle\int_{\mathbb{\mathbb{R}}\setminus\{0\}}  \vert x\vert\, \nu(\mathrm{d}x) \right) t<\infty.\label{eqsecondestimate}
	\end{equation}
	Hence, the integral at the right-hand side of \eqref{eqmamrtingale_tildeM} is well defined for $0\leq t\leq T$. On the other hand, for all $s< t \leq T$, we have
	\begin{multline}
		\mathbb{E}[\tilde{M}^T_t-\tilde{M}^T_s\vert \mathcal{F}^{\eta^T}_{s}]=	\mathbb{E}[\eta^T_t-\eta^T_s\vert \mathcal{F}^{\eta^T}_{s}]-\displaystyle\int_{s}^{t}\mathbb{E}\bigg[\dfrac{X_u-\eta^T_u}{T-u}\bigg\vert \mathcal{F}^{\eta^T}_{s} \bigg]\mathrm{d}u\\-\dfrac{t^2-s^2}{2T}\left(\lambda +\displaystyle\int_{\vert x\vert\geq 1}x\,\nu(\mathrm{d}x)\right).\label{eqShowmartingalegivenFtildebeta}
	\end{multline}
	It follows from \eqref{eqEXtildebeta} that
	\begin{equation}
		\mathbb{E}[\eta^T_t-\eta^T_s\vert \mathcal{F}^{\eta^T}_{s}]=-\dfrac{t-s}{T-s}\eta^T_s+\dfrac{t-s}{T-s}X_s+ \dfrac{t(t-s)}{T}\left(\lambda +\displaystyle\int_{\vert x\vert\geq 1}x\,\nu(\mathrm{d}x)\right).\label{eqEgiventildeBeta_B1} 
	\end{equation}
	From \eqref{eqexpbeta^T} and \eqref{eqexpX} we have
	\begin{align*}
		\mathbb{E}\bigg[\dfrac{X_u-\eta^T_u}{T-u}\bigg\vert \mathcal{F}^{\tilde{\beta}^T}_{s} \bigg]&=\dfrac{\mathbb{E}\left[X_u\vert \mathcal{F}^{X}_{s} \right]}{T}-\mathbb{E}\left[\dfrac{\beta^T_u}{T-u}\bigg\vert \mathcal{F}^{\beta^T}_{s} \right]\\
		&=\dfrac{u-s}{T}\left(\lambda +\displaystyle\int_{\vert x\vert\geq 1}x\,\nu(\mathrm{d}x)\right)+\dfrac{X_s}{T-s}-\dfrac{\beta^T_s}{T-s}.
	\end{align*}
	Hence,
	\begin{align}
		\displaystyle\int_{s}^{t}\mathbb{E}\bigg[\dfrac{ X_u-\eta^T_u}{T-u}\bigg\vert \mathcal{F}^{\eta^T}_{s} \bigg]\mathrm{d}u=\dfrac{t-s}{T-s}(X_s-\eta^T_s)+ \dfrac{(t-s)^2}{2T}\left(\lambda +\displaystyle\int_{\vert x\vert\geq 1}x\,\nu(\mathrm{d}x)\right).\label{eqEgiventildeBeta_B2}
	\end{align}
	Inserting \eqref{eqEgiventildeBeta_B1} and \eqref{eqEgiventildeBeta_B2} into \eqref{eqShowmartingalegivenFtildebeta} we conclude that for all $s< t \leq T$,
	\begin{align*}
		\mathbb{E}[\tilde{M}^T_t\vert \mathcal{F}^{\tilde{\beta}^T}_{s}]=\tilde{M}^T_s.
	\end{align*}
	This completes the proof.
\end{proof}
\subsection{The L\'evy case with Brownian and jump components}

We now consider a L\'evy process with both a Brownian component and a finite
variation jump component. More precisely, let
\[
    Y_t=B_t+X_t,
    \qquad 0\leq t\leq T,
\]
where $B$ is a Brownian motion and $X$ is the finite variation L\'evy process
considered above. We assume that $B$, $X$, and $W$ are independent. Then
\begin{equation}
    Y_t
    =
    B_t
    +
    \left(
        \lambda-\displaystyle\int_{|x|<1} x\,\nu(\mathrm{d}x)
    \right)t
    +
    \sum_{s\leq t}\Delta X_s,
    \qquad 0\leq t\leq T.
    \label{eqdecpurelevy}
\end{equation}
We associate with $Y$ the bridge-type process
\begin{equation}
    \hat{\eta}^T_t
    =W_t-\frac{t}{T}W_T+\frac{t}{T}Y_t,
    \qquad 0\leq t\leq T,
    \label{betahat}
\end{equation}
This setting allows us to combine the Gaussian structure studied in the first part with the jump behaviour arising from the L\'evy component.
\begin{proposition}
	The process $(\hat{M}_t, 0\leq t \leq T)$ given by
	\begin{equation}
		\hat{M}^T_t=\hat{\eta}^T_t-\displaystyle\int_{0}^{t}\dfrac{ Y_s-\hat{\eta}^T_s}{T-s}\mathrm{d}s- \dfrac{t^2}{2T}\left(\lambda +\displaystyle\int_{\vert x\vert\geq 1}x\,\nu(\mathrm{d}x)\right),\label{eqmamrtingale_hatM}
	\end{equation}
	is a martingale with respect to $\mathbb{F}^{\beta^T}\vee\mathbb{F}^{Y}$.
\end{proposition}
\begin{proof}
	We have for all $0\leq t\leq T$,
	\begin{equation}
		\mathbb{E}\bigg[\displaystyle\int_0^t\bigg\vert \dfrac{\hat{\eta}^T_s-Y_s}{T-s} \bigg\vert\mathrm{d}s\bigg]\leq \displaystyle\int_0^t \bigg(\dfrac{\mathbb{E}[\vert\beta^T_s\vert]}{T-s}+\dfrac{\mathbb{E}[\vert X_s\vert]}{T}+\dfrac{\vert\mathbb{E}[\vert B_s\vert]}{T}\bigg)\mathrm{d}s
	\end{equation}
	From \eqref{eqfirstestimate}, \eqref{eqsecondestimate}, and the fact that
	\begin{equation*}
		\displaystyle\int_0^t \mathbb{E}[\vert B_s\vert]\mathrm{d}s\leq\sqrt{\frac{2t}{\pi}}<\infty
	\end{equation*}
	we see that
	\begin{equation}
		\mathbb{E}\bigg[\displaystyle\int_0^t\bigg\vert \dfrac{\hat{\eta}^T_s-Y_s}{T-s} \bigg\vert\mathrm{d}s\bigg]<+\infty.\label{eqhatfinite}
	\end{equation}
	Hence, the integral at the right-hand side of \eqref{eqmamrtingale_hatM} is well defined for $0\leq t\leq T$. On the other hand, for all $s< t \leq T$, we have
	\begin{multline}
		\mathbb{E}[\hat{M}^T_t-\hat{M}^T_s \vert \mathcal{F}^{\beta^T}_{s}\vee \mathcal{F}^{Y}_{s}]=\mathbb{E}[\hat{\eta}^T_t-\hat{\eta}^T_s\vert \mathcal{F}^{\beta^T}_{s}\vee \mathcal{F}^{Y}_{s}]-\displaystyle\int_{s}^{t}\mathbb{E}\bigg[\dfrac{Y_u-\hat{\eta}^T_u}{T-u}\bigg\vert \mathcal{F}^{\beta^T}_{s}\vee \mathcal{F}^{Y}_{s} \bigg]\mathrm{d}u\\-\dfrac{t^2-s^2}{2T}\left(\lambda +\displaystyle\int_{\vert x\vert\geq 1}x\,\nu(\mathrm{d}x)\right).\label{eqShowmartingalegivenFhatbeta}
	\end{multline}
	Simple computation yields
	\begin{equation}
		\mathbb{E}[\hat{\eta}^T_t-\hat{\eta}^T_s\vert \mathcal{F}^{\beta^T}_{s}\vee \mathcal{F}^{Y}_{s}]=\dfrac{t-s}{T-s}(Y_s-\hat{\eta}^T_s)+\dfrac{t(t-s)}{T}\left(\lambda +\displaystyle\int_{\vert x\vert\geq 1}x\,\nu(\mathrm{d}x)\right)\label{eqEgivenhatBeta_B1} 
	\end{equation}
	and
	\begin{align}
		\displaystyle\int_{s}^{t}\mathbb{E}\bigg[\dfrac{Y_u-\hat{\eta}^T_u}{T-u}\bigg\vert \mathcal{F}^{\beta^T}_{s}\vee \mathcal{F}^{Y}_{s} \bigg]\mathrm{d}u=\dfrac{t-s}{T-s}(Y_s-\tilde{\eta}^T_s)+ \dfrac{(t-s)^2}{2T}\left(\lambda +\displaystyle\int_{\vert x\vert\geq 1}x\,\nu(\mathrm{d}x)\right).\label{eqEgivenhatBeta_B2}
	\end{align}
	Inserting \eqref{eqEgivenhatBeta_B1} and \eqref{eqEgivenhatBeta_B2} into \eqref{eqShowmartingalegivenFhatbeta} we conclude that for all $s< t \leq T$,
	\begin{align*}
		\mathbb{E}[\hat{M}^T_t\vert \mathcal{F}^{\beta^T}_{s}\vee \mathcal{F}^{Y}_{s}]=\hat{M}^T_s.
	\end{align*}
	This completes the proof.
\end{proof}
\begin{proposition}
	The process $(\hat{m}_t, 0\leq t \leq T)$ given by 
	\begin{small}
	\begin{multline*}
		\hat{m}^T_t=\hat{\eta}^T_t-\displaystyle\int_{0}^{t}\dfrac{\left(\lambda-\displaystyle\int_{\vert x\vert <1}  x\, \nu(\mathrm{d}x)\right)s+T\sum\limits_{0<u\leq s}\frac{\Delta \hat{\eta}^T_{u}}{u}+\mathfrak{R}^{\hat{\eta}^T}(s,T)-\hat{\eta}^T_s}{T-s}\mathrm{d}s\\-\dfrac{t^2}{2T}\left(\lambda +\displaystyle\int_{\vert x\vert\geq 1}x\,\nu(\mathrm{d}x)\right),\label{eqmamrtingale_tildeM}
	\end{multline*}
	\end{small}
is a martingale with respect to $\mathbb{F}^{\hat{\eta}^T}$, where
	\begin{multline}
	\mathfrak{R}^{\hat{\eta}^T}(s,T)=\dfrac{T-s}{T-s+s\,\tan^{-1}\left( \frac{s}{T}\right)}\bigg[\displaystyle\int_{0}^{s}\bigg(\dfrac{Tu}{(u^2+T^2)}+\tan^{-1}\left( \frac{u}{T}\right)\bigg)\mathrm{d}\hat{\eta}^T_u\\
	-\displaystyle\int_{0}^{s}\displaystyle\int_{\mathbb{R}}\bigg(\dfrac{ Tu^2}{(u^2+T^2)}+u\,\tan^{-1}\left( \frac{u}{T}\right)\bigg)x\,\mathcal{N}(\mathrm{d}u,\mathrm{d}x)\\
    -\displaystyle\int_{0}^{s}\bigg(\dfrac{Tu}{(u^2+T^2)}+\tan^{-1}\left( \frac{u}{T}\right)\bigg)\left(\sum\limits_{0<v\leq u}\frac{\Delta \hat{\eta}^T_{v}}{v}\right)\mathrm{d}u\bigg]\\
	+\dfrac{s\,\tan^{-1}\left( \frac{s}{T}\right)}{T-s+s\,\tan^{-1}\left( \frac{s}{T}\right)}\bigg[\hat{\eta}^T_s
	-\displaystyle\int_{0}^{s}\displaystyle\int_{\mathbb{R}}\dfrac{ux}{T}\,\mathcal{N}(\mathrm{d}u,\mathrm{d}x)-\displaystyle\int_{0}^{s}\left(\sum\limits_{0<v\leq u}\frac{\Delta \hat{\eta}^T_{v}}{v}\right)\mathrm{d}u\bigg]\\
	- \left(\lambda-\displaystyle\int_{\vert x\vert <1}  x\, \nu(\mathrm{d}x)\right)\dfrac{Ts(T-s)+\left(T s^2+T^2(s-T)\right)\tan^{-1}\left( \frac{s}{T}\right)}{T\left[T-s+ s\,\tan^{-1}\left( \frac{s}{T}\right)\right]}.
\end{multline}
Here, $\mathcal{N}$ is the Poisson random measure associated with $X$.
\end{proposition}
\begin{proof}
	Firstly, the process $\hat{\eta}^T$ can be rewritten as
	\begin{equation}
		\hat{\eta}^T_t=\bar{\beta}^T_t+\dfrac{t}{T} X_{t}
	\end{equation}
	As proved in the precious subsection we can show that 
	\begin{equation}
		\mathbb{F}^{\hat{\eta}^T}=\mathbb{F}^{\bar{\beta}^T}\vee \mathbb{F}^{X}.\label{eqFhat}
	\end{equation}
	By \eqref{eqhatfinite} and a well known result of filtering theory, see for instance Theorem 8.1.1 and Remark 8.1.1 in \cite{Kall} or  Proposition 2.30, p. 33 in \cite{BCfiltering}, it follows from \eqref{eqmamrtingale_barM} that the process
	\begin{equation}
		\hat{m}^T_t=\hat{\eta}^T_t-\displaystyle\int_{0}^{t}\dfrac{\mathbb{E}[Y_s\vert \mathcal{F}_s^{\hat{\eta}^T}]-\hat{\eta}^T_s}{T-s}\mathrm{d}s-\dfrac{t^2}{2T}\left(\lambda +\displaystyle\int_{\vert x\vert\geq 1}x\,\nu(\mathrm{d}x)\right)
	\end{equation}
	is an $\mathbb{F}^{\hat{\eta}^T}$-martingale. It remains to compute the conditional expectation $\mathbb{E}[Y_s\vert \mathcal{F}_s^{\hat{\eta}^T}]$. Using \eqref{eqFhat} and \eqref{eqconexpBargicenFbar}, we obtain
	\begin{align*}
		\mathbb{E}[Y_s\vert \mathcal{F}_s^{\hat{\eta}^T}]&=\mathbb{E}[B_s+X_s\vert \mathcal{F}_s^{\bar{\beta}^T}\vee \mathcal{F}_s^{X}]=\mathbb{E}[B_s\vert \mathcal{F}_s^{\bar{\beta}^T}]+X_s\\&=\displaystyle\int_{0}^{s}a(s,u)\mathrm{d}\bar{\beta}^T_u+(\lambda-\displaystyle\int_{\vert x\vert <1}  x\, \nu(\mathrm{d}x))s+T\,\sum\limits_{0<u\leq s}\frac{\Delta \hat{\eta}^T_{u}}{u}.
	\end{align*} 
It remains to derive the expression of 
\begin{equation*}
	\mathfrak{R}^{\hat{\eta}^T}(s,T):=\displaystyle\int_{0}^{s}a(s,u)\mathrm{d}\bar{\beta}^T_u, s<T.
\end{equation*}
	We have 
	\begin{equation*}
		\hat{\eta}^T_u=\bar{\beta}^T_u+\dfrac{u^2}{T} \left(\lambda-\displaystyle\int_{\vert x\vert <1}  x\, \nu(\mathrm{d}x)\right)+\dfrac{u}{T}\sum\limits_{r\leq u}\Delta X_r,\,0\leq u\leq T,
	\end{equation*}
	hence,
	\begin{equation*}
		\mathrm{d}\hat{\eta}^T_u=\mathrm{d}\bar{\beta}^T_u+2  \dfrac{u}{T} \left(\lambda-\displaystyle\int_{\vert x\vert <1}  x\, \nu(\mathrm{d}x)\right)\mathrm{d}u+\left(\sum\limits_{0<r\leq u}\frac{\Delta \hat{\eta}^T_{r}}{r}\right)\mathrm{d}u+\dfrac{u}{T}\displaystyle\int_{\mathbb{R}}x\,\mathcal{N}(\mathrm{d}u,\mathrm{d}x),
	\end{equation*}
	The quantity $\displaystyle\int_{0}^{s}a(s,u)\mathrm{d}\bar{\beta}^T_u$ can be written in terms of the process $\hat{\eta}^T$ as follows:
	\begin{multline}
		\displaystyle\int_{0}^{s}a(s,u)\mathrm{d}\bar{\beta}^T_u=\displaystyle\int_{0}^{s}a(s,u)\mathrm{d}\hat{\eta}^T_u-\dfrac{2}{T} \left(\lambda-\displaystyle\int_{\vert x\vert <1}  x\, \nu(\mathrm{d}x)\right)\displaystyle\int_{0}^{s}u\,a(s,u)\mathrm{d}u\\-\dfrac{1}{T}\displaystyle\int_{0}^{s}\displaystyle\int_{\mathbb{R}}u\,a(s,u)x\,\mathcal{N}(\mathrm{d}u,\mathrm{d}x)-\displaystyle\int_{0}^{s}a(s,u)\left(\sum\limits_{0<r\leq u}\frac{\Delta \hat{\eta}^T_{r}}{r}\right)\mathrm{d}u.
	\end{multline}
	It follows from \eqref{eqa(s,u)} that
	\begin{multline}
		\displaystyle\int_{0}^{s}u\,a(s,u)\mathrm{d}u=\dfrac{T-s}{T-s+ s\,\tan^{-1}\left( \frac{s}{T}\right)}\displaystyle\int_{0}^{s}\bigg(\dfrac{ Tu^2}{(u^2+T^2)}+u\,\tan^{-1}\left( \frac{u}{T}\right)\bigg)\mathrm{d}u\\+\dfrac{\frac{s^3}{2}\,\tan^{-1}\left( \frac{s}{T}\right)}{T-s+ s\,\tan^{-1}\left( \frac{s}{T}\right)}
		=\dfrac{Ts(T-s)+\left(Ts^2+T^2(s-T)\right)\tan^{-1}\left( \frac{s}{T}\right)}{2\left(T-s+ s\,\tan^{-1}\left( \frac{s}{T}\right)\right)},
	\end{multline}
	\begin{multline*}
		\displaystyle\int_{0}^{s}\displaystyle\int_{\mathbb{R}}u\,a(s,u)\,x\,\mathcal{N}(\mathrm{d}u,\mathrm{d}x)=\dfrac{s\,\tan^{-1}\left( \frac{s}{T}\right)}{T-s+s\,\tan^{-1}\left( \frac{s}{T}\right)} \displaystyle\int_{0}^{s}\displaystyle\int_{\mathbb{R}}ux\,\mathcal{N}(\mathrm{d}u,\mathrm{d})x\\
		+\dfrac{T-s}{T-s+s\,\tan^{-1}\left( \frac{s}{T}\right)}\displaystyle\int_{0}^{s}\displaystyle\int_{\mathbb{R}}\bigg(\dfrac{T u^2}{(u^2+T^2)}+u\tan^{-1}\left( \frac{u}{T}\right)\bigg)x\,\mathcal{N}(\mathrm{d}u,\mathrm{d}x),
	\end{multline*}
and
	\begin{multline*}
		\displaystyle\int_{0}^{s}a(s,u)\left(\sum\limits_{0<v\leq u}\frac{\Delta \hat{\eta}^T_{v}}{v}\right)\mathrm{d}u=\dfrac{s\,\tan^{-1}\left( \frac{s}{T}\right)}{T-s+s\,\tan^{-1}\left( \frac{s}{T}\right)}\displaystyle\int_{0}^{s}\left(\sum\limits_{0<v\leq u}\frac{\Delta \hat{\eta}^T_{v}}{v}\right)\mathrm{d}u\\
		+\dfrac{T-s}{T-s+s\,\tan^{-1}\left( \frac{s}{T}\right)}\displaystyle\int_{0}^{s}\bigg(\dfrac{Tu}{(u^2+T^2)}+\tan^{-1}\left( \frac{u}{T}\right)\bigg)\left(\sum\limits_{0<v\leq u}\frac{\Delta \hat{\eta}^T_{v}}{v}\right)\mathrm{d}u.
	\end{multline*}
Combining this we obtain the desired result.
\end{proof}
\begin{proposition}
	For any $0<t<T$, we have
\begin{multline}
	\mathbb{E}[Y_T\vert \mathcal{F}^{\hat{\eta}^T}_{t}]=\displaystyle\int_0^t\frac{T \left((T^2 + s^2) \tan^{-1}\left(\frac{s}{T}\right) + s (T - s)\right)}{(T^2 + s^2) \left(s \tan^{-1}\left(\frac{s}{T}\right) + T - s\right)}\mathrm{d}\hat{m}^T_s\\+\displaystyle\int_0^t\frac{T^2(T - s)}{(T^2 + s^2) \left(s \tan^{-1}\left(\frac{s}{T}\right) + T - s\right)}\displaystyle\int_{\mathbb{R}}x\,\left(\mathcal{N}(\mathrm{d}s,\mathrm{d}x)-\,\nu(\mathrm{d}x)\mathrm{d}s\right).
\end{multline}
\end{proposition}
\begin{proof}
	For $0<t<T$, we have
	\begin{equation*}
		\mathbb{E}[Y_{T}\vert \mathcal{F}^{\hat{\eta}^T}_{t}]=\mathbb{E}[B_{T}+X_T\vert \mathcal{F}^{\bar{\beta}^T}_{t}\vee \mathcal{F}^{X}_{t}]=\mathbb{E}[B_{T}\vert \mathcal{F}^{\bar{\beta}^T}_{t}]+\mathbb{E}[X_T\vert  \mathcal{F}^{X}_{t}]
	\end{equation*}
	For the first expectation, we have
	\begin{multline*}
		\mathbb{E}[ B_{T}\vert \mathcal{F}^{\bar{\beta}^T}_{t}]=\displaystyle\int_{0}^{t}\mathfrak{a}(t,u)\mathrm{d}\bar{\beta}^T_u=\dfrac{T-t}{T-t+t\,\tan^{-1}\left( \frac{t}{T}\right)}\displaystyle\int_{0}^{t}\bigg(\dfrac{Tu}{(u^2+T^2)}+\tan^{-1}\left( \frac{u}{T}\right)\bigg)\mathrm{d}\bar{\beta}^T_u\\+\dfrac{ t\,\tan^{-1}\left( \frac{t}{T}\right)}{T-t+t\,\tan^{-1}\left( \frac{t}{T}\right)}\bar{\beta}^T_t\\
		=\dfrac{T-t}{T-t+t\,\tan^{-1}\left( \frac{t}{T}\right)}\left( \displaystyle\int_{0}^{t}\bigg(\dfrac{Tu}{(u^2+T^2)}+\tan^{-1}\left( \frac{u}{T}\right)\bigg)\mathrm{d}\bar{\beta}^T_u-\bar{\beta}^T_t\right)+\bar{\beta}^T_t.
	\end{multline*}
	For the second expectation, we have
	\begin{equation}
		\mathbb{E}[X_T\vert \mathcal{F}^{X}_{t}]=\left(\lambda +\displaystyle\int_{\vert x\vert\geq 1}x\,\nu(\mathrm{d}x)\right)(T-t)+X_t.
	\end{equation}
Setting, 
	\begin{equation}
		Y_t^T=\mathbb{E}[Y_{T}\vert \mathcal{F}^{\hat{\eta}^T}_{t}],\quad t\in [0,T],
	\end{equation}
	by It\^o formula we obtain
	\begin{multline}
		\mathrm{d}Y_t^T=\frac{T \left((T^2 + t^2) \tan^{-1}\left(\frac{t}{T}\right) + t (T - t)\right)}{(T^2 + t^2) \left(t \tan^{-1}\left(\frac{t}{T}\right) + T - t\right)}\mathrm{d}\hat{\beta}^T_t-\frac{\left((T^2 + t^2) \tan^{-1}\left(\frac{t}{T}\right) + t (T - t)\right)}{(T^2 + t^2) \left(t \tan^{-1}\left(\frac{t}{T}\right) + T - t\right)}X_t\,\mathrm{d}t\\
		-\frac{\left((T^2 + t^2) \tan^{-1}\left(\frac{t}{T}\right) + t (T - t)\right)}{(T^2 + t^2) \left(t \tan^{-1}\left(\frac{t}{T}\right) + T - t\right)}t\,\mathrm{d}X_t
		\\
		+\left(\bar{\beta}^T_t-\displaystyle\int_{0}^{t}\left(\dfrac{Tu}{T^2+u^2}+\tan^{-1}\Big(\frac{u}{T}\Big)\right)\mathrm{d}\bar{\beta}^T_u\right)\frac{T \left((T^2 + t^2) \tan^{-1}\left(\frac{t}{T}\right) + t (T - t)\right)}{(T^2 + t^2) \left(t \tan^{-1}\left(\frac{t}{T}\right) + T - t\right)^2}\mathrm{d}t\\
		-\left(\lambda +\displaystyle\int_{\vert x\vert\geq 1}x\,\nu(\mathrm{d}x)\right)\mathrm{d}t+\mathrm{d}X_t.
	\end{multline}
	Using the fact that
	\begin{equation}
		\mathrm{d}\hat{\eta}^T_t=\mathrm{d}\hat{m}^T_t+\dfrac{X_t+\mathfrak{R}^{\hat{\eta}^T}(t,T)-\hat{\eta}^T_t}{T-t}\mathrm{d}t+\dfrac{t}{T}\left(\lambda +\displaystyle\int_{\vert x\vert\geq 1}x\,\nu(\mathrm{d}x)\right)\mathrm{d}t
	\end{equation}
	we obtain
	\begin{multline}
		\mathrm{d}Y_t^T=\frac{T \left((T^2 + t^2) \tan^{-1}\left(\frac{t}{T}\right) + t (T - t)\right)}{(T^2 + t^2) \left(t \tan^{-1}\left(\frac{t}{T}\right) + T - t\right)}\mathrm{d}\hat{m}^T_t+\frac{T \left((T^2 + t^2) \tan^{-1}\left(\frac{t}{T}\right) + t (T - t)\right)}{(T^2 + t^2) \left(t \tan^{-1}\left(\frac{t}{T}\right) + T - t\right)}\dfrac{X_t}{T-t}\mathrm{d}t\\
		+\frac{T \left((T^2 + t^2) \tan^{-1}\left(\frac{t}{T}\right) + t (T - t)\right)}{(T^2 + t^2) \left(t \tan^{-1}\left(\frac{t}{T}\right) + T - t\right)^2}\left(\displaystyle\int_{0}^{t}\left(\dfrac{Tu}{T^2+u^2}+\tan^{-1}\Big(\frac{u}{T}\Big)\right)\mathrm{d}\bar{\beta}^T_u-\bar{\beta}^T_t\right)\mathrm{d}t\\
		+\frac{T \left((T^2 + t^2) \tan^{-1}\left(\frac{t}{T}\right) + t (T - t)\right)}{(T^2 + t^2) \left(t \tan^{-1}\left(\frac{t}{T}\right) + T - t\right)}\dfrac{\bar{\beta}^T_t}{T-t}\mathrm{d}t-\frac{T \left((T^2 + t^2) \tan^{-1}\left(\frac{t}{T}\right) + t (T - t)\right)}{(T^2 + t^2) \left(t \tan^{-1}\left(\frac{t}{T}\right) + T - t\right)}\dfrac{\hat{\beta}^T_t}{T-t}\mathrm{d}t\\
		+\frac{\left((T^2 + t^2) \tan^{-1}\left(\frac{t}{T}\right) + t (T - t)\right)}{(T^2 + t^2) \left(t \tan^{-1}\left(\frac{t}{T}\right) + T - t\right)}t\left(\lambda +\displaystyle\int_{\vert x\vert\geq 1}x\,\nu(\mathrm{d}x)\right)\mathrm{d}t\\-\frac{\left((T^2 + t^2) \tan^{-1}\left(\frac{t}{T}\right) + t (T - t)\right)}{(T^2 + t^2) \left(t \tan^{-1}\left(\frac{t}{T}\right) + T - t\right)}X_t\,\mathrm{d}t
		-\frac{\left((T^2 + t^2) \tan^{-1}\left(\frac{t}{T}\right) + t (T - t)\right)}{(T^2 + t^2) \left(t \tan^{-1}\left(\frac{t}{T}\right) + T - t\right)}t\,\mathrm{d}X_t
		\\
		+\left(\bar{\beta}^T_t-\displaystyle\int_{0}^{t}\left(\dfrac{Tu}{T^2+u^2}+\tan^{-1}\Big(\frac{u}{T}\Big)\right)\mathrm{d}\bar{\beta}^T_u\right)\frac{T \left((T^2 + t^2) \tan^{-1}\left(\frac{t}{T}\right) + t (T - t)\right)}{(T^2 + t^2) \left(t \tan^{-1}\left(\frac{t}{T}\right) + T - t\right)^2}\mathrm{d}t\\
		-\left(\lambda +\displaystyle\int_{\vert x\vert\geq 1}x\,\nu(\mathrm{d}x)\right)\mathrm{d}t+\mathrm{d}X_t.
	\end{multline}
	Thus, 
	\begin{multline}
		\mathrm{d}Y_t^T=\frac{T \left((T^2 + t^2) \tan^{-1}\left(\frac{t}{T}\right) + t (T - t)\right)}{(T^2 + t^2) \left(t \tan^{-1}\left(\frac{t}{T}\right) + T - t\right)}\mathrm{d}\hat{m}^T_t\\+\frac{T^2(T - t)}{(T^2 + t^2) \left(t \tan^{-1}\left(\frac{t}{T}\right) + T - t\right)}\displaystyle\int_{\mathbb{R}}x\,\left(\mathcal{N}(\mathrm{d}t,\mathrm{d}x)-\,\nu(\mathrm{d}x)\mathrm{d}t\right).
	\end{multline}
	This completes the proof.
\end{proof}


	
\end{document}